\documentclass[a4paper,11pt]{article}

\usepackage{mathtools}
\usepackage[utf8]{inputenc}
\usepackage{amsthm,amsmath,amstext, amssymb, xcolor, scrextend, tikz, multirow, enumerate, wasysym, makecell, geometry, tabularx, hyperref,  pbox, lipsum,csquotes,  tikz-cd, multicol, calc, accents, stackengine, kpfonts}
\usepackage[all]{xy}
\usepackage{xcolor}
\usepackage{palatino}
\usepackage{cancel}
\usepackage{authblk}
\usepackage{fullpage}

\DeclareFontFamily{U}{mathx}{\hyphenchar\font45}
\DeclareFontShape{U}{mathx}{m}{n}{<-> mathx10}{}
\DeclareSymbolFont{mathx}{U}{mathx}{m}{n}
\DeclareMathAccent{\widebar}{0}{mathx}{"73}

\newcommand{\dbl}{[\hspace{-0.2ex}[}
\newcommand{\dbr}{]\hspace{-0.2ex}]}
\newcommand{\db}[1]{\dbl {#1} \dbr}

\newcommand{\invlim}{\underleftarrow{\textnormal{lim}}\,}
\def\lim{\mathop{\rm lim}\nolimits}

\def\HH{\mathrm{HH}}

\def\Tor{\mathrm{Tor}}

\def\Ker{\mathrm{Ker}}
\def\Ker{\mathrm{Ker}}

\newcommand{\ctens}{\widehat{\otimes}}

\newcommand{\tn}{\textnormal}

\newcommand{\iso}{\cong}

\numberwithin{equation}{section}
\newtheorem{theorem}{Theorem}[section]
\newtheorem{prop}[theorem]{Proposition}
\newtheorem{lemma}[theorem]{Lemma}

\newtheorem{corol}[theorem]{Corollary}

\theoremstyle{remark}
\newtheorem{remark}[theorem]{Remark}
\theoremstyle{definition}
\newtheorem{defn}[theorem]{Definition}
\newtheorem{example}[theorem]{Example}

\usetikzlibrary{decorations.shapes}
\usetikzlibrary{decorations.text}
\usetikzlibrary{calc}

\hypersetup{colorlinks=true,linkcolor=[rgb]{0.5, 0.0, 0},citecolor=[rgb]{0.5, 0.0, 0.5}, filecolor=magenta,urlcolor=blue}

\title{Homological properties of extensions of algebras}

\author[a]{Kostiantyn Iusenko}
\author[b]{John William MacQuarrie}
\affil[a]{Instituto de Matem\'{a}tica e Estat\'{i}stica, Univ. de S\~{a}o Paulo, S\~{a}o Paulo, SP, Brazil}
\affil[b]{Universidade Federal de Minas Gerais, Belo Horizonte, MG, Brazil}

\begin{document}

\footnotetext{\textit{Email addresses:} iusenko@ime.usp.br (Kostiantyn Iusenko), john@mat.ufmg.br (John MacQuarrie)}

\maketitle

\begin{abstract}
    We consider a class of extensions of associative algebras, which we refer to as ``strongly proj-bounded extensions''.  We prove that the finiteness of the left global dimension and the support of the Hoch\-schild homology is preserved by strongly proj-bounded extensions, generalizing results of Cibils, Lanzillota, Marcos and Solotar.  Moreover, we show that the finiteness of the big left finitistic dimension is preserved by strongly proj-bounded extensions.  In order to construct examples, we describe a new class of extensions of algebras of finite relative global dimension, which may be of independent interest.  The results apply both for abstract  (meaning no topology) and pseudocompact algebras.

\medskip


\end{abstract}

\section{Introduction}

Let $k$ be an algebraically closed field.  In the series of recent papers (see \cite{CLMS20,CLMS20-,CLMS20Arx, CLMS21} and references therein) Cibils, Lanzilotta, Marcos and Solotar consider what they call ``bounded'' extensions of associative $k$-algebras: that is, extensions of algebras $B\subseteq A$ such that $A/B$ has finite projective dimension as a $B$-bimodule, such that $A/B$ is projective as either a left or a right $B$-module, and such that some tensor power $(A/B)^{\otimes_B p}$ is $0$.
The authors compare the homological properties of the algebras in such an extension: namely, they show that the left global dimension of $B$ is finite if, and only if, the left global dimension of $A$ is finite, and that the Hochschild homology groups of $B$ vanish in high enough degree if, and only if the Hochschild homology groups of $A$ vanish in high enough degree.  In particular, when $B\subseteq A$ is a bounded extension of finite dimensional algebras, then Han's conjecture (see \cite{Han06}) holds for $B$ if, and only if, it holds for $A$. 

We generalize bounded extensions as follows.  Say that an extension of $k$-algebras $B\subseteq A$ is  \textit{strongly proj-bounded} if: 
    \begin{enumerate}
    \item $A/B$ has finite projective dimension as a $B$-bimodule;
    \item $A/B$ is projective as either a left or a right $B$-module; 
    \item There is $p\geqslant 1$ such that $A/B^{\otimes_B n}$ is projective as a 
    $B$-bimodule
    for any $n\geqslant p$; 
    \item The $(A^e,B^e)$-projective dimension of  $A$ is finite. 
\end{enumerate}
We have two main goals.  Firstly, we prove that such extensions satisfy the same homological preservation properties as cited above; secondly, we prove that $B$ has finite finitistic dimension if, and only if, $A$ does (Theorem \ref{Theorem.HHhom} and Theorem \ref{theorem abstract preservation of gd and fd}).  As example corollary, one obtains \cite[Theorem A]{GPS} as a consequence of Theorem \ref{theorem abstract preservation of gd and fd}.  The extra generality of the extensions we consider over bounded extensions has two immediate advantages: firstly, it allows the extension $B\leqslant A$ to be `` large'', in the sense that the tensor algebra $T[B,A/B]$ can have arbitrarily long non-zero words $v_1\otimes\hdots \otimes v_s$ ($s\in \mathbb{N}, v_i\in A/B$); secondly, even for finite dimensional algebras, it allows extensions for which $A$ has more primitive idempotents than $B$, and hence can assist in reductions of Han's Conjecture and the Finitistic Dimension Conjecture.

We prove that all the results above are true in the category of pseudocompact algebras.  Pseudocompact algebras, a class of topological algebra, are among the most natural generalization of finite dimensional algebras, inheriting almost all the powerful formal properties of finite dimensional algebras.  They are dual to coalgebras and hence our results have dual versions for coalgebras.  That said, pseudocompact algebras have both informal (eg.\ multiplication is more tractable than comultiplication) and formal (eg.\ Pontryagin duality) advantages over coalgebras.  In nature, pseudocompact algebras and their modules are the objects one should study to understand the representation theory of profinite groups.

The paper is organized as follows: in Section 
\ref{Section pc algebras} we recall the basic notions of pseudocompact algebras and their pseudocompact modules. In Section \ref{Section.Relative Homological algebra} we collect the basic properties of relative projective modules and corresponding resolutions as in \cite{Hoch56}; we also develop a general construction of non-trivial extensions of algebras $B\subseteq A$ with finite relative global dimension (Section \ref{Section Ext.finite.gldim}) that will be used in Section \ref{Section projbounded}  to construct examples of finite dimensional strongly proj-bounded extensions. We believe that Corollary \ref{corol weakly acyclic has finite relgd} of Theorem \ref{theorem algebras with finite relgd} is of independent interest.
In Section \ref{Section. RelHomPC} we define Hochschild homology for pseudocompact algebras and develop the notion of relative projective modules. Section \ref{Section projbounded} is devoted to proj-bounded extensions, giving several examples.  A key ingredient of the cited work of Cibils, Lanzilotta, Marcos and Solotar 
is the so-called Jacobi-Zariski long exact sequence of a bounded extension, which relates the Hochschild homology of $A$, the Hochschild homology of $B$, and the relative Hochschild homology of the extension (see also \cite{Kaygun,KaygunEr}). In Section \ref{Section. JZ sequences} we give a 
Jacobi-Zariski long exact sequence for extensions satisfying Conditions 1 to 3 of a strongly proj-bounded extension (extensions we call simply ``proj-bounded''), in both the abstract and pseudocompact cases.  Finally, in Section \ref{Section.HomProperties} we prove that the finitude of the left global dimension, left big finitistic dimension, and support of the Hochschild homology, are preserved by strongly proj-bounded extensions.  The supporting results in Section \ref{Section.HomProperties} are of quite a general nature and we expect them to have further use.

\textbf{Acknowledgements.} We would like to thank Eduardo N. Marcos for stimulating discussion and a number of important comments about the manuscript. We also thank Changchang Xi who made several important remarks  about relative homological dimension, and pointed out to us that the proof of Corollary \ref{corol weak acyclic finite relgldim as bimod} was not as easy as we'd thought!  We also thank the anonymous referee, whose attentive reading has improved several parts of the article: in particular, the elegant proof of Theorem \ref{Theorem.HHhom} is due to the referee.  The first author was partially supported by FAPESP grant 2018/23690-6 and by CNPq Universal grant 405540/2023-0.   The second author was partially supported by CNPq Universal grants 402934/2021-0 and 401998/2023-1, CNPq Produtividade 1D grant 303667/2022-2, and FAPEMIG Universal grant APQ-00971-22.

\section{Pseudocompact algebras}\label{Section pc algebras}

Let $k$ be an algebraically closed field (thought of as a discrete topological ring).  A pseudocompact $k$-algebra is an inverse limit of finite dimensional associative $k$-algebras, taken in the category of topological algebras -- see for instance \cite{Brumer,IM23} for an introduction to pseudocompact objects.  Morphisms in the category of pseudocompact algebras are continuous algebra homomorphisms.  Pseudocompact algebras arise in several natural contexts: completed group algebras of profinite groups, the objects of study in the representation theory of profinite groups, are pseudocompact (cf.\ \cite{Brumer}); the algebras dual to coalgebras are precisely the pseudocompact algebras (see for instance \cite[Theorem 3.6]{Simson11}).  

Throughout the text, in order to avoid confusion with topological objects, we use the word ``abstract'' to refer to objects without a topology: thus ``abstract algebra'' just means ``algebra''.  Our main results apply to extensions of abstract algebras, as well as to extensions of pseudocompact algebras.  In this article, the main use of pseudocompact algebras will be for the construction of examples: the completed path algebra of a quiver (defined below) is a pseudocompact algebra.  Working with the abstract path algebra $kQ$ of a quiver (even a finite quiver, if it has cycles) is notoriously tricky.  But working with the completed path algebra $k\db{Q}$, many technical difficulties do not arise ($k\db{Q}$ is unital, projective $k\db{Q}$-modules are what one would hope, the Jacobson radical of $k\db{Q}$ is what one would hope, etc.).  The reader whose interest is in finite dimensional or abstract algebras will miss nothing by skipping mentions to pseudocompact algebras.

\begin{defn}
Let $Q$ be a finite quiver.  The \emph{completed path algebra} (cf.\ for instance \cite{DerksonWeymanZelevinsky}) is defined to be the cartesian product
$$k\db{Q} := \prod_{n=0}^\infty kQ_n,$$
where $kQ_n$ is the vector space with basis the paths of $Q$ of length exactly $n$.  Each $kQ_n$ is given the discrete topology and $k\db{Q}$ is given the product topology.  The multiplication of paths is as in the abstract path algebra -- this multiplication extends to a continuous multiplication on $k\db{Q}$.

Now let $Q$ be an arbitrary quiver.  We may consider $Q$ as the union (that is, direct limit) of its finite subquivers $Q_i$.  The operation ``completed path algebra of a finite quiver'' $k\db{-}$ is a contravariant functor (see also \cite{IM20}) to the category of pseudocompact algebras (if $\iota$ is an inclusion of finite quivers, $k\db{\iota}$ sends a path of the larger quiver to itself if it is contained in the smaller, and to $0$ otherwise).    The \emph{completed path algebra of $Q$} is
$$k\db{Q} := {\invlim_{i}} k\db{Q_i}.$$
The category of pseudocompact algebras is closed under taking inverse limits, so $k\db{Q}$ is pseudocompact.
\end{defn}

A general reference for the majority of the claims in the following paragraphs is \cite{Brumer}.  Let $A$ be a pseudocompact algebra.  A \emph{pseudocompact left $A$-module} is an inverse limit of finite dimensional topological left $A$-modules each given the discrete topology, the limit being taken in the category of topological $A$-modules -- when we refer to an $A$-module, we implicitly mean left $A$-module.  Morphisms in the category of pseudocompact $A$-modules are continuous module homomorphisms -- this category is abelian with exact inverse limits.  Pseudocompact right modules and bimodules are defined analogously. The category of pseudocompact left $A$-modules is dual to the category of discrete right $A$-modules (that is, to the category of topological $A$-modules with the discrete topology).  In particular, if $A$ is finite dimensional then the category of pseudocompact $A$-modules is dual to the category of abstract $A$-modules, though for general $A$ they are not the same.  We work exclusively with pseudocompact modules in this article.

Given a pseudocompact algebra $A$, a pseudocompact right $A$-module $U = \invlim U_i$ and a pseudocompact left $A$-module $V = \invlim V_j$, the tensor product $U\otimes_A V$ need not be pseudocompact (see \cite[Proposition 2.2]{MacQuarrieSymondsZalesskii} for situations where it is).  The \emph{completed tensor product} $U\ctens_A V$ is by definition $\invlim_{i,j} U_i\otimes_A V_j$, a pseudocompact vector space.  The completed tensor product operation behaves exactly as one would expect in the category of pseudocompact modules: $-\ctens_A V$ is right exact, $A\ctens_A V \iso V$, etc.  The Tor functors $\tn{Tor}_n^A(U,V)$ are defined as one would expect \cite[\S 2]{Brumer}.

Let $B$ be a closed subalgebra of the pseudocompact algebra $A$.  If $V$ is a pseudocompact $B$-module, the \emph{induced} $A$-module is $A\ctens_B V$, with action from $A$ on the left factor.  If $U$ is a pseudocompact $A$-module, the \emph{restriction} of $U$ to $B$ is simply $U$ with multiplication restricted to $B$.  The restriction of $U$ can be expressed as $\tn{Hom}_A(A,U)$, where $A$ is being treated as an $A - B$-bimodule in the obvious way.  Since the functor $A\ctens_B-$ is left adjoint to the functor $\tn{Hom}_A(A,-)$ \cite[Section 2.2]{MacQuarrieSymondsZalesskii}, it follows that induction is left adjoint to restriction.

A free pseudocompact $A$-module is precisely a direct product of copies of the $A$-module $A$.  A projective pseudocompact $A$-module is precisely a continuous direct summand of a free module (``continuous'' here means that both the inclusion and projection maps are continuous, or equivalently that the corresponding complement is closed).  The category of pseudocompact $A$-modules has projective covers \cite[Chapter II, Theorem 2]{GabrielThesis}.  The functor $-\ctens_A V$ is exact if, and only if, $V$ is projective.  The indecomposable projective left $A$-modules are precisely modules isomorphic to $Ae$, where $e$ is a primitive idempotent of $A$.  In particular, when $Q$ is a quiver, the projective left $k\db{Q}$-modules have the form $k\db{Q}e$, where $e$ is the stationary path at some vertex of $Q$.  Analogous statements hold for right $A$-modules.  Given two pseudocompact algebras $A, B$, the pseudocompact vector space $A\ctens_k B^{op}$ (where $B^{op}$ denotes the opposite algebra to $B$) has the structure of a pseudocompact algebra:
$$(a\ctens b)\cdot (a'\ctens b') := aa'\ctens b'b.$$
In particular, the pseudocompact algebra $A$ has a pseudocompact enveloping algebra $A^e := A\ctens_k A^{op}$.
The category of pseudocompact $A$-bimodules is equivalent, in just the same way as with abstract algebras, to the category of pseudocompact left $A^e$-modules, and to the category of pseudocompact right $A^e$-modules.

Finally, given a pseudocompact algebra $B$ and a pseudocompact $B$-bimodule $M$, the corresponding \emph{completed tensor algebra} is defined to be
$$T_B\db{M} := \prod_{n\in \{0,1,2,\hdots\}}M^{\ctens_B n},$$
where $M^{\ctens_B 0} = B, M^{\ctens_B 1} = M$ and for $n>1$, $M^{\ctens_B n} = M\ctens_B M\ctens_B\hdots\ctens_B M$ ($n$ times).  Endowing this product with the product topology, the completed tensor algebra is a pseudocompact algebra \cite[Section 7.5]{Gabriel73}.

\section{Relative homological algebra}\label{Section.Relative Homological algebra}

\subsection{Relative projective modules, resolutions and relative global dimension}\label{subsection abstract rel proj definitions}

In order to describe relative homological algebra (cf.\, \cite{Hoch56}), we first recall the notion of relatively projective modules.
Given an extension $B \subseteq A$ of algebras, an $A$-module $M$ is called \emph{relatively $B$-projective}, or $(A, B)$-\textit{projective},
if it satisfies either of the following equivalent conditions:
\begin{enumerate}
\item $M$ is isomorphic to a direct summand of the induced module $A\otimes_{B} V$, for $V$ some left $B$-module;
    \item if ever an $A$-module homomorphism onto $M$ splits as a $B$-module homomorphism, then it splits as an $A$-module homomorphism.
\end{enumerate}

For a proof of the equivalence of these conditions, and several other characterizations of relatively projective modules, see for instance \cite[Section 1]{Hoch56} and \cite[Section 1]{Th85}.  In the special case where $B=k$ is the unique subalgebra of $A$ of dimension $1$, the notion of $(A,B)$-projective modules coincides with the notion of projective modules.
At the other end of the spectrum, every $A$-module is $(A,A)$-projective. A projective $A$-module $P$ is always $(A, B)$-projective, and every $A$-module of the form $A \otimes_{B} N$ with $N$ a left $B$-module, is $(A, B)$-projective.  
An exact sequence of $A$-module homomorphisms
$$\cdots\to  M_{n+1}\xrightarrow{f_{n+1}} M_{n}\xrightarrow{f_{n}} M_{n-1}\to 0$$
(some $n\in \mathbb{Z}$) is called $(A, B)$-\textit{exact} if, for each $i\geqslant n$, the kernel of $f_{i}$ is a direct $B$-module summand of $M_{i}$ (cf.\ \cite[Section 1]{Hoch56}).  One may check that a sequence of morphisms 
$\{f_i\,|\,i\geqslant n\}$ is
$(A,B)$-exact if, and only if,
\begin{enumerate}
    \item $f_{i} \circ f_{i+1}=0$ for all $i\geqslant n$, 
    \item 
    there exists a \emph{contracting $B$-homotopy}: that is, a sequence of $B$-module homomorphisms $h_{i}: M_{i} \rightarrow M_{i+1}$ ($i\geqslant n-2$)  such that $f_{i+1} h_{i}+h_{i-1}f_{i}$ is the identity map on $M_{i}$.
\end{enumerate}

One may now develop the concepts of relative projective dimension and relative global dimension. 
Given an $A$-module $M$, we define the \textit{relative projective} dimension of $M$ to be the minimal number $n$, denoted by $\tn{pd}_{(A,B)} M$, such that there is an $(A,B)$-exact sequence
$$
0 \rightarrow P_{n} \rightarrow \cdots \rightarrow P_{1} \rightarrow P_{0} \rightarrow M \rightarrow 0.
$$
with the $P_i$ $(A,B)$-projective.  If such an exact sequence does not exist, the relative projective dimension of $M$ is infinite. 
This definition is equivalent to the corresponding definition from \cite[Section 2]{XiXu}. The \textit{left relative global dimension} $\tn{gldim}(A, B)$ of the extension $B \subseteq A$ is the supremum of the relative projective dimension of the $A$-modules, if this number exists, and infinity otherwise (``left'' because our modules are left modules).

Having relative projective resolutions, the relative derived functors $\tn{Tor}^{(A,B)}_n$ and $\tn{Ext}_{(A,B)}^n$ can be defined, and we refer to \cite{Hoch56,BuHo61} for the details. Recall that given an $A$-bimodule $M$, the Hochschild homology of $A$ with coefficients in $M$ is defined as:
 $$
     \HH_*(A,M)=\Tor_*^{A^e}(M,A).
$$
In particular, when $M=A$ the Hochschild homology of $A$ is defined as $\HH_*(A):=\HH_*(A,A)$. Considering the extension of enveloping algebras $B^{e} \subseteq A^{e}$, the relative Hochschild homology is defined  as:
$$
\HH_{*}(A \mid B, M)=\operatorname{Tor}_{*}^{(A^{e}, B^{e})}(M, A).
$$
These spaces were originally defined in  \cite{Hoch56} -- see \cite[Section 2]{CLMS20-} for the equivalence of the definitions.

We present a useful relative analogue of a well-known result about projective modules:

\begin{lemma}\label{lemma after relprojdim kernel is summand}
Let $A$ be an algebra and $B$ a subalgebra. Suppose that $M$ is an $A$-module of relative projective dimension  $\tn{pd}_{(A,B)}M=d$ and that 
$$\cdots \rightarrow F_1 \rightarrow F_{0} \rightarrow M \rightarrow 0$$ 
is a relative projective resolution of $M$.
Then $e\geqslant d$ if, and only if, the kernel of $F_{e} \rightarrow F_{e-1}$ is a direct summand of $F_e$ as an $A$-module.
\end{lemma}

\begin{proof} It is clear that if the kernel of $F_{e} \rightarrow F_{e-1}$ is an $A$-module direct summand of $F_e$, then $e\geqslant d$.  We prove the other direction, mimicking the proof for normal projective resolutions. We proceed by induction on $d$ and on $e$. If $d=0$, then $M$ is relative projective. By the relative projectivity of $M$ and the existence of a $B$-contracting homotopy, the kernel $K$ of $F_0\to M$ is an $A$-module direct summand of $F_0$.   This finishes the proof if $e=0$, and if $e>0$, then replacing $M$ by $K$ we decrease $e$ and obtain the result by induction.
Assume now that $d>0$. Let 
$$0 \rightarrow P_{d} \rightarrow P_{d-1} \rightarrow \ldots \rightarrow P_{0} \rightarrow M \rightarrow 0$$ be a minimal relative projective resolution of $M$.
Now, by the relative version of Schanuel's Lemma 
 \cite[p.\ 1538]{Th85}, 
 we have $P_{0} \oplus \Ker\left(F_{0} \rightarrow M\right) \cong F_{0} \oplus \Ker\left(P_{0} \rightarrow M\right)$.  But the latter module has relative projective resolution
 $$
0 \rightarrow P_{d} \rightarrow P_{d-1} \rightarrow \ldots \rightarrow P_{2} \rightarrow P_{1} \oplus F_{0} \rightarrow \Ker\left(P_{0} \rightarrow M\right) \oplus F_{0} \rightarrow 0
$$
of length $d-1$, and hence $P_{0} \oplus \Ker\left(F_{0} \rightarrow M\right)$ has relative projective dimension at most $d-1$.
It follows that the induction hypothesis applies to the sequence
$$
F_{e} \rightarrow F_{e-1} \rightarrow \ldots \rightarrow F_{2} \rightarrow P_{0} \oplus F_{1} \rightarrow P_{0} \oplus \Ker\left(F_{0} \rightarrow M\right) \rightarrow 0
$$
of length $e-1$: if $s_i: F_i\to F_{i+1}$ are (all but the last of) the maps of the $B$-contracting homotopy for the resolution of $F_i$'s, then a $B$-contracting homotopy $t_i$ for the new sequence is defined by
\begin{align*}
    t_0 : P_{0} \oplus \Ker\left(F_{0} \rightarrow M\right) & \to P_{0} \oplus F_{1} \\
    (x,y) & \mapsto (x, s_0(y)),
\end{align*}
\begin{align*}
    t_1 : P_{0} \oplus F_{1} & \to F_2 \\
    (x,z) & \mapsto s_1(z),
\end{align*}
and $t_i = s_i$ for $i\geqslant 2$.
It follows from the induction hypothesis that $\Ker\left(F_{e} \rightarrow F_{e-1}\right)$ is a direct summand of $F_e$, as required.
\end{proof}

If $B\subseteq A$ is an extension of finite dimensional algebras and $U$ is a finitely generated $A$-module, the $(A,B)$-\emph{relative projective cover} of $U$ is an $(A,B)$-projective module $P$ together with a surjective $B$-split $A$-module homomorphism $\rho:P\to U$, minimal in the sense that no proper direct summand of $P$ surjects onto $U$ via $\rho$.  Relative projective covers exist and are unique up to isomorphism \cite[Proposition 1.3]{Th85}.

\subsection{Extensions of finite relative global dimension}\label{Section Ext.finite.gldim}

When we define strongly proj-bounded extensions $B\subseteq A$ in Section \ref{Section projbounded}, we will require among other things that $\tn{pd}_{(A^e, B^e)}A$ be finite, which will trivially be the case when $\tn{gldim}(A^e,B^e)$ is finite.  There are several papers in the literature that construct extensions of algebras $B\subseteq A$ with prescribed properties for the relative global dimension $\tn{gldim}(A,B)$. For instance, in \cite{Gr75,XiXu,EHIS04}, extensions of relative global dimension $0$ and $1$ are discussed. 
In \cite{Guo18}, the author presents a general method for constructing non-trivial extensions of algebras of relative global dimension at least $n$, for each given $n$.  Here we construct a class of extensions of algebras whose relative global dimension is finite, which in Section \ref{Section projbounded} will be used to construct non-trivial finite dimensional proj-bounded extensions.  We believe that Corollary \ref{corol weakly acyclic has finite relgd} is of independent interest.

\begin{theorem}\label{theorem algebras with finite relgd}
 Let $Q$ be a quiver together with a partition $V_1,V_2,\dots, V_n$ of the vertex set of $Q$, satisfying the properties:
 \begin{enumerate}
   \item There are no arrows from a vertex in $V_i$ to a vertex in $V_j$ if $i<j$;
   \item There are no arrows between distinct vertices in each $V_i$.
 \end{enumerate}
 Let $A = kQ/I$, with $I$ admissible, and let $B$ be a subalgebra of $A$ generated as a subalgebra by the elements of a quiver $R$ which satisfies the following properties:
 \begin{enumerate}[i.]
    \item Each vertex of $R$ is a sum of vertices of $Q$;
    \item The arrows $\beta$ of $R$ are linear combinations of paths in $Q$ and have the property that for any vertex $e$ of $Q$, there are vertices $h,g$ of $Q$ such that $\beta e = h\beta e$ and $e\beta = e\beta g$.
    \item For each vertex $e$ of $Q$ and loop $\gamma$ of $Q$ at $e$, there is an element $\beta$ of $B$ such that $\gamma = e\beta$.
\end{enumerate} 
Then the relative global dimension $\tn{gldim}(A,B)$ is at most $n-1$.
\end{theorem}

We make some elementary observations regarding Condition ii.\ that will help us in the proof: let $\beta$ be an arrow of $R$, hence a linear combination of paths of $A$, possibly with different start and end points in $Q$. If $e\beta\neq 0$, the condition $e\beta = e\beta g$ says that $e\beta$ has a ``unique start point'' $g$.  Then also $\beta g = e\beta g$, since otherwise $\beta g$ would not satisfy the symmetric condition.  The image to have in mind is that $\beta$ can be the sum of the paths
$$\xymatrix{\bullet & \bullet & \bullet\ar[l] & \bullet\ar@/_1pc/[lll]}$$
but not the sum of the paths
$$\xymatrix{\bullet & \bullet\ar[l] & \bullet\ar@/_1pc/[ll]}$$
Suppose that $e_1\beta = e_1\beta g_1$ and $e_2\beta = e_2\beta g_2$ are non-zero.  Then (since $I$ is admissible)
$$e_1 = e_2 \,\Longleftrightarrow\, e_1\beta = e_2\beta \,\Longleftrightarrow\, e_1\beta g_1 = e_2\beta g_2 \,\Longleftrightarrow\, \beta g_1 = \beta g_2 \,\Longleftrightarrow\, g_1 = g_2.$$

\begin{proof}
The theorem follows easily from $n-1$ applications of the following principal claim: If $U$ is a left $A$-module supported on vertices in $V_1\cup \hdots\cup V_m$ (some $1\leqslant m\leqslant n$) and $P\to U$ is a relative projective cover of $U$, then its kernel is supported on vertices in $V_1\cup \hdots \cup V_{m-1}$.  Our task is to prove the principal claim, so fix such a $U$.

\medskip

Consider first the natural map $\rho : A\otimes_B U\to U$.  We will show that $A\otimes_B U = A_S\oplus A_D$, where
 \begin{align*}
 A_S &= \sum_{e\in Q_0} Ae \otimes_{B} eU,\\
 A_D &= \sum_{e\neq f\in Q_0} Ae \otimes_{B} fU.
 \end{align*}
 The notation $Ae \otimes_{B} eU$ is being used abusively to refer to the subspace of $A\otimes_B U$ generated by elements of the form $ae\otimes eu$, with $a\in A$ and $u\in U$, and similarly with $Ae \otimes_{B} fU$.  It's clear that $A_S$ and $A_D$ are $A$-submodules of $A\otimes_B U$ and that their sum is $A\otimes_B U$.  We need to check that their intersection is $0$. So fix an element 
 $$x = \sum_s a_se_s \otimes_{B} e_su_s = 
 \sum_t a_te_t \otimes_{B} f_tu_t$$ 
 of the intersection, where the first expression is written in terms of $A_S$ and the second is written in terms of $A_D$ (so $e_t\neq f_t$).  Consider $A\otimes_k U$ and the subspaces
  \begin{align*}
 \widetilde{A_S} &= \sum_{e\in Q_0} Ae \otimes_{k} eU,\\
 \widetilde{A_D} &= \sum_{e\neq f\in Q_0} Ae \otimes_{k} fU.
 \end{align*}
 It's clear that $A\otimes_k U = \widetilde{A_S}\oplus\widetilde{A_D}$. Denote by $Y, Z$ the elements 
 $$Y = \sum_s a_se_s \otimes_{k} e_su_s \hbox{ and } Z = \sum_t a_te_t \otimes_{k} f_tu_t$$ 
 in $\widetilde{A_S}, \widetilde{A_D}$ respectively.  We have by the construction of the tensor product over $B$ that
 $$Y - Z = \sum_{r\in C} a_re_rb_r\otimes f_ru_r
  - a_re_r\otimes b_rf_ru_r$$
  wherein we may suppose that each $a_r$ is a non-zero scalar multiple of a path in $A$ and each $b_r$ is a path in $R$.  Condition ii.\ of our hypotheses implies that the element $e_rb_r = e_rb_rg$ for some vertex $g\in Q_0$, and hence the element $a_re_rb_r\otimes f_ru_r$ is in either $\widetilde{A_S}$ (if $g = f_r$) or $\widetilde{A_D}$ (if $g\neq f_r$), and similarly with $a_re_r\otimes b_rf_ru_r$.  We claim both elements are in the same one.  Write $e_rb_r = e_rb_rg$ and $b_rf_r = hb_rf_r$ for $g,h\in Q_0$.  Supposing everyone is non-zero (the claim is trivial otherwise) we have by the observations before the proof that $e_r = h \Longleftrightarrow g = f_r$.  This is precisely the claim.  Hence, there is a subset $E$ of $C$ such that  
  $$Y = \sum_{r\in E} a_re_rb_r\otimes f_ru_r
  - a_re_r\otimes b_rf_ru_r.$$
  Thus $x = 0$, completing the proof that $A\otimes_{B} U = A_S\oplus A_D$.
 
 \medskip
 
 It is clear that the restriction $\rho|_{A_S} : A_S \to U$ is surjective.  It thus remains to check that its kernel is supported on vertices in $V_1\cup \hdots \cup V_{m-1}$.  Certainly $A_S$ is supported on $V_1\cup \hdots \cup V_{m}$, because $U$ is and there are no arrows from $V_1\cup \hdots \cup V_{m}$ to $V_{m+1}\cup \hdots \cup V_{n}$ by hypothesis.  Denote by $E$ the sum of the vertices in $V_m$.  By the conditions on the quiver $Q$, it's enough to check that the map $A_S\cap EAE\otimes_B U\to U$ is injective.  Since there are no arrows between distinct vertices in $V_m$, we have 
 $$A_S\cap EAE\otimes_B U = \bigoplus_{e\in V_m} eAe\otimes eU$$
 and hence it's enough to check that the map $eAe\otimes eU \to U$ is injective for each $e\in V_m$.  Fix an element $\sum_r e\delta_r e \otimes eu_r\in \tn{Ker}(\rho)$, so that $\sum_r e\delta_ru_r = 0$.  Each $e\delta_re$ is a linear combination of products of loops at $e$ and hence by condition iii.\ we can write $e\delta_re = \sum_s e\beta_{rs}$ with $\beta_{rs}\in B$.  Now
 \begin{align*}
     \sum_r e\delta_re\otimes eu_r = \sum_r\sum_s e\beta_{rs}\otimes eu_r = \sum_r\sum_s e\otimes \beta_{rs}eu_r & = \sum_r e\otimes e\delta_ru_r  
      = e\otimes(\sum_r e\delta_ru_r) = 0.
 \end{align*}
 \end{proof}

The quivers $Q$ whose vertices can be partitioned as in Theorem \ref{theorem algebras with finite relgd} are precisely those quivers that are acyclic once the loops are removed.  Algebras $A$ of the form given in Theorem \ref{theorem algebras with finite relgd} have been studied in several places (e.g. \cite{CMMP97,MMP00} and references therein), as they are precisely the finite dimensional algebras whose indecomposable projective modules can be labelled as $P_i$ in such a way that there are no $A$-module homomorphisms from $P_j\to P_i$ when $i$ is strictly less that $j$.  Such algebras are referred to in \cite{CoelPlatz} as ``weakly triangular'' algebras.

The following is a a relative version of the well-known result that any finite dimensional algebra whose quiver is acyclic has finite global dimension.  Recall that if ever $B$ is a subalgebra of $A$ and $z$ is a unit of $A$, we may consider the conjugate subalgebra ${}^zB = \{zbz^{-1}\,|\,b\in B\}$.  If ever $V$ is a $B$-module, we obtain a conjugate ${}^zB$-module $z(V)$, whose elements we denote by $zv$, for $v\in V$.  The multiplication is defined as ${}^zb\cdot zv := zbv$ for $b\in B, v\in V$.  Given a homomorphism of $B$-modules $\alpha:U\to V$ we define a homomorphism of ${}^zB$-modules $z(\alpha) : z(U)\to z(V)$ by $z(\alpha)(zu) := z\alpha(u)$.  The conjugation operation $z(-)$ can thus be treated as an exact functor from the category of $B$-modules to the category of ${}^zB$-modules.  If $V$ is a $B$-module, one has an isomorphism of $A$-modules $z(A\otimes_B V)\iso A\otimes_{{}^zB}z(V)$ given by $z(a\otimes v)\mapsto az^{-1}\otimes zv$.  Putting this together, one may check that $\tn{gldim}(A,B) = \tn{gldim}(A, {}^zB)$ for any unit $z$ of $A$.  Note that every unital subalgebra $B$ of a bounded path algebra $kQ/I$ is conjugate, by \cite{Malcev42}, to a subalgebra having a complete set of primitive idempotents that are sums of vertices of $Q$.

\begin{corol}\label{corol weakly acyclic has finite relgd}
Let $Q$ be a quiver with vertices $1,\hdots,n$ such that there are no arrows $i\to j$ when $i$ is strictly smaller than $j$, and let $A = kQ/I$, with $I$ admissible.  Let $B$ be conjugate to a subalgebra of $A$ generated as a subalgebra by the elements of a quiver $R$ satisfying the following properties: 
\begin{itemize}
    \item The vertices of $R$ are sums of vertices of $Q$;
    \item The arrows $\beta$ of $R$ are such that $\beta = f\beta e$ where $e,f$ are vertices of $Q$;
    \item $R$ contains every loop of $Q$.
\end{itemize}
Then $\tn{gldim}(A,B)\leqslant n-1$.
\end{corol}

\begin{proof}
This is immediate from Theorem \ref{theorem algebras with finite relgd}, taking $V_i = \{i\}$. 
\end{proof}

\begin{example}
Consider $A=kQ/I$, where $Q$ is
$$\xymatrix{1\ar[rd]_{\gamma}\ar[rr]^{\alpha} && 2\ar@(rd,ru)[]_{\delta}\ar[ld]^{\beta} \\
& 3 & }$$
and $I$ is an admissible ideal of $kQ$. Let $B$ be the subalgebra of $A$ generated by the following quiver $R$:
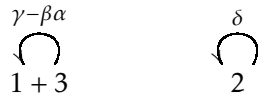
$$\xymatrix{1+3 \ar@(ru,lu)[]_{\gamma-\beta\alpha} && 2\ar@(ru,lu)[]_{\delta}}$$
Then $B=kR/\langle I\cap kR, (\gamma-\beta\alpha)^2\rangle$. It is easily seen that the conditions of Corollary \ref{corol weakly acyclic has finite relgd} are satisfied and hence $\tn{gldim}(A,B)\leqslant 2$.
\end{example}

\begin{example}
One might wonder whether a more general phenomenon is true: if a finite dimensional algebra $A$ has finite global dimension and $B$ is a subalgebra, must $A$ have finite relative global dimension with respect to $B$?  The answer is no, as the following example shows: let $A$ be the algebra $kQ/I$, where $Q$ is
$$
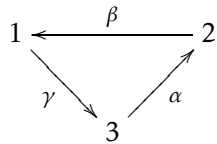
\xymatrix{1\ar[rd]_{\gamma} && 2\ar[ll]_{\beta} \\
& 3\ar[ru]_{\alpha} & }$$
and $I$ is generated by $\beta\alpha$ and $\alpha\gamma$.  Then $\tn{gldim}(A) = 3$.  Let $B$ be the subalgebra with basis $\langle e_1, e_2, e_3, \gamma\beta\rangle$.  The $B$-relative projective dimension of every simple module is infinite.  There are similar examples starting with a cycle of length $n$ for every $n\geqslant 2$.
\end{example}

\begin{example}
Theorem \ref{theorem algebras with finite relgd} is considerably more general than Corollary \ref{corol weakly acyclic has finite relgd}.  For example, Theorem \ref{theorem algebras with finite relgd} says that if $A$ is the path algebra of 
$$\xymatrix{2 \ar@(lu,ld)[]_{\gamma} && 1\ar@(rd,ru)[]_{\beta}\ar[ll]_{\alpha}}$$
modulo the ideal generated by $\beta^2$ and $\gamma^2$, and $B$ is the subalgebra with basis $\{e_1+e_2\,,\, \beta+\gamma\}$, then $\tn{gldim}(A,B)$ is at most 1.
\end{example}

\begin{corol}\label{corol weak acyclic finite relgldim as bimod}
Let $A,B$ be as in Corollary \ref{corol weakly acyclic has finite relgd}.  The relative global dimension $\tn{gldim}(A^e, B^e)$ is at most $2n-2$.
\end{corol}

\begin{proof}
We make use of the description $Q'$ of the quiver of $A^e$ given in \cite[Proposition 2.2]{SkEnv}.  Its vertices are pairs $(i,j)$ with $i,j\in \{1, \hdots,n\}$, where $1,\hdots,n$ are the vertices of $Q$.  For each arrow $\alpha:a\to b$ in $A$, there are $n$ arrows $(\alpha,i) : (a,i)\to (b,i)$ and $n$ arrows $(i,\alpha) : (i,b) \to (i,a)$ in the quiver of $A^e$. 

We partition the vertices.  For each $s\in \{1-n, \dots, n-1\}$ define the set 
$$V_s = \{(i,j)\,|\,i-j = s\}.$$
One may check that there are no arrows from a vertex of $V_a$ to a vertex of $V_b$ when $a<b$, and that the only arrows between vertices in a given $V_a$ are loops, so Conditions $1,2$ of Theorem \ref{theorem algebras with finite relgd} are satisfied.  

We turn now to the conditions on the subalgebra.  Note that if $B$ and $C = {}^zB$ are conjugate subalgebras of $A$, then ${C}^e = {}^{z\otimes z^{-1}}(B^e)$ and so $B^e, C^e$ are conjugate subalgebras of $A^e$.  We may thus suppose that $B$ is equal, rather than just conjugate, to the subalgebra described in Corollary \ref{corol weakly acyclic has finite relgd}. If $x,y$ are vertices of $R$, then $x = \sum i, y = \sum j$ with each $i, j$ a vertex of $Q$.  The required quiver $R'$ for $B^e$ will have vertices of the form $(x,y) := \sum_{i,j}(i,j)$, a sum of vertices of $Q'$, as required in Condition $i.$  We may define the arrows of $R'$ to be $(\beta, x) := \sum_{i}(\beta,i)$ and $(x,\beta) := \sum_i(i,\beta)$, with $x = \sum i$ a vertex of $R$ and $\beta$ an arrow of $R$.  Condition $ii.$ is now easily checked, as is the final condition on loops.
\end{proof}

\begin{remark}

The principal obstruction to the construction of deeper examples of proj-bounded extensions seems to arise from the surprisingly small amount of literature on the subject of relative homological algebra for associative algebras.  By contrast, an enormous amount of work has been done on relative homological algebra for finite groups -- such results are hugely influencial in the modular representation theory and block theory of finite groups (see for instance \cite{Benson, Lin1, Lin2}).  We are not the first to suggest that the development of a robust relative homological algebra for associative algebras may yield great rewards.
\end{remark}

\subsection{Relative homological algebra over pseudocompact algebras} \label{Section. RelHomPC}

\subsubsection{Hochschild homology for pseudocompact algebras}

Hochschild cohomology for coalgebras was introduced by Doi in \cite{DoiHom}. Using the standard duality between coalgebras and pseudocompact algebras 
(see, for instance, \cite[Theorem 3.6]{Simson11}) one can thus define Hochschild homology for pseudocompact algebras by dualizing these definitions.  We choose to work directly with pseudocompact algebras, as it is not any harder.  Given a pseudocompact algebra $A$ and pseudocompact bimodule $M$ (which recall from Section \ref{Section pc algebras} is the same thing as a pseudocompact $A^e = A\ctens A^{\rm op}$-module) we define Hochschild homology with coefficients in $M$ as 
 $$
     \HH_*(A,M)=\Tor_*^{A^e}(M,A).
 $$
The Hochschild homology of $A$ is defined as 
$$
 \HH_*(A):=\HH_*(A,A)=\Tor_*^{A^e}(A,A).
 $$ 
 As with abstract algebras, the Hochschild homology of $A$ can be computed using the so-called standard Hochschild resolution: for a pseudocompact algebra $A$ consider the positively graded $A$-bimodule $C_{*}^{\prime}(A)$ defined for $q \in \{1,2,\hdots\}$ by
 $$
 C_{q}^{\prime}(A)=A \ctens A^{\ctens q} \ctens A.
 $$
 For $q=0$ set $C_{0}^{\prime}(A)=A \ctens A .$ The vector space $C_{q}^{\ctens}(A)$ is a pseudocompact $A$-bimodule for each $q$. Now  define $d: C_{q}^{\prime}(A) \rightarrow C_{q-1}^{\prime}(A)$ on pure tensors by
 $$
 d_{i}\left(a_{0} \ctens a_{1} \ctens \cdots \ctens a_{q} \ctens a_{q+1}\right)=\sum_{i=0}^q(-1)^i a_{0} \ctens \cdots\ctens a_{i} a_{i+1} \ctens\cdots \ctens a_{q+1}.
 $$
 The maps $d: C_{q}^{\prime}(A) \rightarrow C_{q-1}^{\prime}(A)$ are continuous morphisms of $A$-bimodules satisfying $d^2=0$. Define also $s$ :
 $C_{q}^{\prime}(A) \rightarrow C_{q+1}^{\prime}(A)
 $ by
 $$
 s\left(a_{0} \ctens a_{1} \ctens \cdots \ctens a_{q} \ctens a_{q+1}\right)=1 \ctens a_{0} \ctens a_{1} \ctens \cdots \ctens a_{q} \ctens a_{q+1}.
 $$
 One checks that $ds+sd=\mathrm{id}$, and hence the complex $\left(C_{*}^{\prime}(A), b^{\prime}\right)$ is a resolution of $A$ by free $A$-bimodules. 

 By definition of the Tor groups, $H_{*}(A, M)$ are the homology groups of the chain complex
 $$
 \left(M \ctens_{A \ctens A^{\text {op }}} C_{*}^{\prime}(A), \text { id } \ctens d \right).
 $$
 One can simplify these complexes using the isomorphism
 $$
 \varphi: M \ctens_{A \ctens A^{\text{op}}} C_{q}^{\prime}(A) \rightarrow C_{q}(A, M):=M \ctens A^{\ctens q}
 $$
 defined by $\varphi\left(m \ctens a_{0} \ctens a_{1} \ctens \cdots \ctens a_{q} \ctens a_{q+1}\right)=a_{q+1} m a_{0} \ctens a_{1} \ctens \cdots \ctens a_{q}$. Passing the
 differential $d$ through these isomorphisms,  we obtain the differential $b: C_{q}(A, M) \rightarrow C_{q-1}(A, M)$ given by
 $$
 \begin{aligned}
 b\left(m \ctens a_{1} \ctens \cdots \ctens a_{q}\right)=& m a_{1} \ctens \cdots \ctens a_{q} \\
 &+\sum_{i=1}^{q-1}(-1)^{i} m \ctens a_{1} \ctens \cdots \ctens a_{i} a_{i+1} \ctens \cdots \ctens a_{q} \\
 &+(-1)^{q} a_{q} m \ctens a_{1} \ctens \cdots a_{q-1}.
 \end{aligned}
 $$

 The following proposition shows that presenting any pseudocompact algebra $A$ as an inverse limit of finite dimensional algebras $A=\invlim A_i$, its Hochschild homology can be calculated via the usual Hochschild homology of the algebras $A_i$.

 \begin{prop}\label{prop.HomologyLimit}
 Suppose that $A$ is a pseudocompact algebra. The graded vector space $\HH_*(A)$ can be calculated as 
 \begin{equation}\label{eq.HHinvLim}
     \HH_*(A)=\invlim_{I} \HH_*(A/I),
 \end{equation}
 where $I$ runs through the open ideals in $A$.
 \end{prop}

 \begin{proof}
 This result follows from a sequence of checks so we just give a sketch.  First observe that a morphism of algebras $\rho:A_j\to A_i$ induces a morphism of complexes $C_{*}^{\prime}(A_j) \to C_{*}^{\prime}(A_i)$, and hence a map on their homologies.  It follows that presenting $A$ as the inverse limit of the inverse system $\{A/I, \rho_{IJ}:A/J\to A/I\}$, we obtain an inverse system of graded vector spaces $\HH_*(A/I)$.

 The argument used to prove \cite[Proposition 6.5.7]{RZ} shows that the inverse limit of this inverse system is $\HH_*(A)$. 
 \end{proof}

 It is important to observe that the maps of the above inverse system of $\tn{HH}_*(A/I)$ are not necessarily surjective.  Indeed, $\HH_n(A)$ can be zero even if each $\HH_n(A_i)$ is non-zero.  For instance, taking $A=k[[x]]$ we have that $k[[x]]=\invlim k[x]/x^i$. For each $i$ the homology of $A_i=k[x]/x^i$ is non-zero in every positive degree, while $\HH_*(k[[x]])$ is concentrated in degrees $0$ and $1$, as the following example shows.

 \begin{example} 
Consider $A=k[[x_1,x_2,\dots,x_n]]$. One can mimic  the arguments from \cite[Exercise 9.1.3]{Weibel} to calculate $\HH_*(A)$.
We obtain
 $$\HH_*(k[[x_1,x_2,\dots,x_n]])\cong k[[x_1,x_2,\dots,x_n]]\,\ctens\,\Lambda^*(k^n),
 $$
 where $\Lambda^*(k^n)$ is the exterior algebra of the vector space $k^n$.  In particular,  $\HH_*(k[[x_1,x_2,\dots,x_n]])$ is non-zero in degrees $0,\hdots,n$ and zero otherwise.

For finite dimensional algebras, a result of Keller \cite[Theorem 2.2]{Keller98} says that if the global dimension of a finite dimensional algebra $A$ is finite, then $\tn{HH}_*(A)$ is supported in degree $0$.  These examples show that this is not the case for pseudocompact algebras, as the power series ring $k[[x_1,x_2,\dots,x_n]]$ has global dimension $n$ by \cite[Theorem 1.12]{AusBuch58}.
 \end{example}

\subsubsection{Relative Hochschild homology for pseudocompact algebras}

We modify the content of Section \ref{subsection abstract rel proj definitions} in the obvious way so it makes sense for  pseudocompact algebras and modules. If $A$ is a pseudocompact algebra and $B$ is a closed subalgebra of $A$, we say that $B\subseteq A$ is an extension of pseudocompact algebras.

\begin{lemma}\label{lemma pc relproj char}
Let $M$ be a pseudocompact $A$-module.  The following are equivalent:
\begin{enumerate}
    \item $M$ is isomorphic to a continuous direct summand of the induced $A$-module $A\ctens_B V$, for some pseudocompact $B$-module $V$.
    \item Any continuous surjective $A$-module homomorphism $f\colon U\to M$ that splits continuously as a $B$-module homomorphism, also splits continuously as an $A$-module homomorphism.  
\end{enumerate}
\end{lemma}

\begin{proof}
That 2 implies 1 is easy: the multiplication map $A\ctens_B M\to M$ clearly splits as a $B$-module homomorphism, via the continuous map sending $m$ to $1\ctens m$.  So by 2, $M$ is isomorphic to a continuous direct summand of $A\ctens_B M$ and 1 holds.

To prove that 1 implies 2, our situation is as follows
 $$\xymatrix{
   & A\ctens_B V \ar@/_/[d]_{\pi} \\
 U\ar@/_/[r]_{f} & M\ar@/_/[u]_{\iota}\ar@/_/@{-->}[l]_{\gamma}}$$
 where solid arrows are continuous $A$-module homomorphisms, dashed arrows are continuous $B$-module homomorphisms, and $f\gamma = \pi\iota = \tn{id}_M$.  We are looking for a continuous $A$-module homomorphism $\gamma':M\to U$ such that $f\gamma' = \tn{id}_M$. With abstract modules, the calculation can be done with elements, but with the completed tensor product it is easier to work formally.  Denote by $\varepsilon, \eta$ the counit and unit of the induction-restriction adjunction described in Section \ref{Section pc algebras}.  The map $\delta = \varepsilon_U(1\ctens \gamma\pi\eta_V)$ is an $A$-module homomorphism from $A\ctens_B V$ to $U$ and we claim that $\gamma' = \delta\iota$ is the required splitting of $f$.  By naturality of $\varepsilon$, functoriality of induction, naturality of $\varepsilon$ again and the counit-unit equations, we have
 \begin{align*}
     f\gamma' & = f\varepsilon_U(1\ctens \gamma\pi\eta_V)\iota \\
     & = \varepsilon_M(1\ctens \pi\eta_V)\iota \\
     & = \varepsilon_M(1\ctens\pi)(1\ctens\eta_V)\iota \\
     & = \pi\varepsilon_{A\ctens_BV}(1\ctens\eta_V)\iota \\
     & = \pi\iota \\
     & = \tn{id}_M.
 \end{align*}
\end{proof}

We say that $M$ is \emph{relatively $B$-projective} or \emph{$(A,B)$-projective} if it satisfies the conditions of Lemma \ref{lemma pc relproj char}.

Now just as with abstract modules one can construct $(A,B)$-projective resolutions of pseudocompact modules and define relative $\Tor_*^{(A,B)}$ in direct analogy with the abstract case.
The $B$-relative Hochschild  homology of the pseudocompact algebra $A$ with coefficients in the pseudocompact $A$-bimodule $M$ is
   $$\HH_*(A\mid B, M)=\Tor_*^{(A^e,B^e)}(M,A).$$
As in abstract case, the above definition is equivalent to the one using $(A^e,B\ctens A^{op})$-projective resolutions.

\section{Proj-bounded extensions}\label{Section projbounded}

\subsection{Definitions and first examples}

We define the extensions of interest to us and provide several examples.

\begin{defn}\label{def proj bounded}
We say that the extension of $k$-algebras $B\subseteq A$ is \emph{proj-bounded} if it satisfies the following three conditions:
 \begin{enumerate}
     \item $A/B$ is of finite projective dimension as a $B^e$-module
     \item $A/B$ is projective as either a left or a right $B$-module.
     \item There exists a natural number $p\geqslant 1$ (called the \textit{index of projectivity}) such that the tensor power $A/B^{\otimes_B n}$ is projective as a $B^e$-module, for any $n\geqslant p$.
\end{enumerate}
We say that the extension is \emph{strongly proj-bounded} if it satisfies the additional condition that
\begin{itemize}
    \item[4.] $A$ has finite $(A^e,B^e)$-projective dimension.
\end{itemize}
\end{defn}

The definitions of proj-bounded and strongly proj-bounded extensions of pseudocompact algebras are the same: one must only replace abstract tensor products $\otimes_k, \otimes_B$ with completed tensor products $\ctens_k, \ctens_B$ throughout.

\begin{example}
A bounded extension of algebras is clearly proj-bounded, since if $A/B^{\otimes_B p} = 0$, then $A/B^{\otimes_B n} = 0$ is projective for any $n\geqslant p$.  A bounded extension is in fact strongly proj-bounded: this follows from the $(A^e,B^e)$-projective resolution
\cite[Theorem 2.3]{CLMS20-} of $A$, whose length is at most $p$ when $A/B^{\otimes_B p} = 0$. 
\end{example}

 \begin{example}
The motivating example for the development of bounded extensions comes from \cite{CLMS20}.  The authors begin with a finite dimensional bounded path algebra $B = kQ/I$ ($Q$ a finite quiver and $I$ an admissible ideal) and add to $Q$ a set $F$ of what they call ``inert arrows''.  If the induced algebra $B_F$ (see \cite[Definition 3.3]{CLMS20}) is finite dimensional, then the extension $B\subseteq B_F$ is bounded.  Applying the same procedure to completed path algebras and arbitrary sets $F$ of arrows, the corresponding extension $B\subseteq B_F$ will of course not be bounded, but it is strongly proj-bounded: the extension is proj-bounded because $M = \prod_{a\in F}Bt(a)\ctens_k s(a)B$ is a projective pseudocompact $B$-bimodule, while Condition 4 of Definition \ref{def proj bounded} is satisfied because the finite sequence (2.1) from \cite{CLMS20} remains a relative projective resolution for $B_F$ as a $B^e$-module.
 \end{example}
 
 \begin{example}
 The smallest example of an extension that is proj-bounded but not strongly proj bounded is $A = k[x]/x^2, B = k$.  The conditions of a proj-bounded extension are trivially satisfied because every $k$-bimodule is projective.  But $A$ has infinite ($B^e$-relative) projective dimension, so the extension is not strongly proj-bounded. 
 \end{example}

 \begin{example}
We give an example demonstrating that a tensor power of $A/B$ can be projective as a bimodule without $A/B$ being projective as a bimodule.  Again we work with the completed path algebra.  Let $Q$ be the following quiver:
$$\xymatrix{
4 & 3\ar[l]^{\delta} & 2\ar[l]^{\gamma}\ar@(ru,lu)[]_{\beta} & 1\ar[l]^{\alpha}
.}$$
We introduce a notation that we will use also in other examples.  For any vertex $e$ of any quiver $Q$, denote by $S_e$ (resp.\ $\widetilde{S_e}$) the simple left (resp.\ right) $B = k\db{Q}$-module at vertex $e$, and by $P_e$ (resp.\ $\widetilde{P_e}$) the projective left (resp.\ right) $B$-module at vertex $e$.

Returning to the specific example, consider the $B = k\db{Q}$-bimodule $M = X\oplus P$, where
$$X = S_3\ctens_k \widetilde{S_1}\,,\, P = P_2\ctens_k \widetilde{P}_2.$$
Then $P$ is projective as a $B$-bimodule, while $X$ is projective as a right $B$-module but not as a left $B$-module.  From this setup there are two obvious algebras one may construct:
\begin{itemize}
    \item $A' = B\oplus M$, the trivial extension algebra of $B$ by $M$ (see \cite[Chapter III.2]{ARS97}).  We check the conditions to conclude that $B\subseteq A'$ is proj-bounded:
    \begin{enumerate}
        \item To see that $A'/B\iso M$ has finite projective dimension as a $B^e$-module, it is enough to check that $X$ has finite projective dimension as a $B$-bimodule.  But the projective cover of $X$ is $P_3\ctens_k \widetilde{P_1}$ and its kernel is isomorphic to $S_4\ctens \widetilde{S_1}$, which is projective as a bimodule.
        
        \item Immediate, since both $X$ and $P$ are projective as right $B$-modules.
        
        \item $M$ is not projective as a bimodule, but 
        $$M\ctens_B M \iso \cancel{X\ctens_B X}\oplus \cancel{X\ctens_B P}\oplus \cancel{P\ctens_B X} \oplus P\ctens_B P = P\ctens_B P$$
        is a projective $B$-bimodule, and similarly with higher tensor powers.
    \end{enumerate}
    The extension $B\subseteq A'$ is however not strongly proj-bounded.

    \item $A = T_B\db{M}$.  Similar checks to those above show that the extension $B\subseteq A$ is proj-bounded.  This extension is strongly proj-bounded: the exact sequence from the proof of Theorem 2.5 in \cite{CLMS20}, interpreted for pseudocompact algebras, shows that the relative projective dimension of a completed tensor algebra is finite.
\end{itemize}

\end{example}

\subsection{Finite dimensional strongly proj-bounded extensions}

There are strongly proj-bounded extensions $B\subseteq A$ when $A$ is finite dimensional, that are not bounded.  We present a class of examples.  The results apply to more general pseudocompact algebras, but we prove them for finite dimensional algebras to maintain focus.  We first describe a construction yielding proj-bounded extensions (Lemma \ref{lemma matrix AoverB proj as bimod}) and then ``intersect'' this construction with the construction of Corollary \ref{corol weakly acyclic has finite relgd} to obtain a class of finite dimensional strongly proj-bounded extensions.  We provide an explicit example of such an extension.

\medskip

Let $A_1, \widebar B, A_2$ be finite dimensional algebras.  Let $M_1$ be a finitely generated $\widebar B-A_1$-bimodule and $M_2$ be a finitely generated $A_2-\widebar B$-bimodule.  Let $M_{21}$ be an $A_2-A_1$-bimodule, together with a bimodule homomorphism $\rho : M_2\otimes_{\widebar B} M_1 \to M_{21}$.
Define $A$ to be the algebra 
$$A = \begin{pmatrix}A_2 & M_2 & M_{21} \\
0 & \widebar B & M_1 \\
0 & 0 & A_1\end{pmatrix}$$
with the obvious multiplication: that is, first multiply the matrices and then interpret the entries of the product matrix in the natural way -- if $m_1\in M_1, m_2\in M_2$ then $m_2\cdot m_1$ is interpreted as $\rho(m_2\otimes m_1)$ in $M_{21}$.  One easily checks that $A$ is an associative algebra.
Denote by $E$ the idempotent $\begin{pmatrix}1_{A_2} & 0 & 0 \\
0 & 0 & 0 \\
0 & 0 & 1_{A_1}\end{pmatrix}$ and by $B$ the subalgebra 
$$\begin{pmatrix}0 & 0 & 0 \\
0 & \widebar B & 0 \\
0 & 0 & 0\end{pmatrix} \times \langle E\rangle$$
of $A$.

\begin{lemma}\label{lemma matrix AoverB proj as bimod}
With the setup as above, $A/B$ is projective as a $B$-bimodule if, and only if, $M_1$ is projective as a left $\widebar B$-module and $M_2$ is projective as a right $\widebar B$-module.
\end{lemma}

\begin{proof}
Suppose that $M_1, M_2$ are projective as left and right $\widebar B$-modules, respectively, so there are left and right $\widebar B$-modules $T_1, T_2$ respectively such that $M_1\oplus T_1 = \widebar B^n$ (as a left $\widebar B$-module) and $M_2\oplus T_2 = \widebar B^m$ (as a right $\widebar B$-module).  Then $\begin{pmatrix}0 & T_2 & 0 \\
0 & 0 & T_1 \\
0 & 0 & 0\end{pmatrix}$ is a $B$-bimodule and we have a decomposition of $B$-bimodules as follows:
\begin{align*}
 &    A/B \oplus 
    \begin{pmatrix}0 & T_2 & 0 \\
0 & 0 & T_1 \\
0 & 0 & 0\end{pmatrix} \\ 
& = \begin{pmatrix}A_2 & \widebar B^m & M_{21} \\
0 & 0 & \widebar B^n \\
0 & 0 & A_1\end{pmatrix} /\langle E\rangle
 \\
& =  \begin{pmatrix}A_2 & 0 & M_{21} \\
0 & 0 & 0 \\
0 & 0 & A_1\end{pmatrix} /\langle E\rangle \oplus \begin{pmatrix}0 & \widebar B^m & 0 \\
0 & 0 & 0 \\
0 & 0 & 0\end{pmatrix}\oplus \begin{pmatrix}0 & 0 & 0 \\
0 & 0 & \widebar B^n \\
0 & 0 & 0\end{pmatrix}.
\end{align*}
The first summand is a direct sum of copies of the projective bimodule $P_E\otimes_k \widetilde{P_E}$, the second is a direct sum of $m$ copies of the projective bimodule $P_E\otimes_k \widebar{B}$, and the third is a direct sum of $n$ copies of the projective bimodule $\widebar{B}\otimes_k\widetilde{P_E}$.

For the converse, observe that 
the condition that $A/B$ be projective as a $B$-bimodule implies that both 
$$
\begin{pmatrix}0 & M_2 & 0 \\
0 & 0 & 0 \\
0 & 0 & 0\end{pmatrix},\quad 
\begin{pmatrix}0 & 0 & 0 \\
0 & 0 & M_1 \\
0 & 0 & 0\end{pmatrix},\quad 
$$
are projectives $B$-bimodules, so that in particular $M_2$ is a projective right $\widebar B$-module and $M_1$ is a projective left $\widebar B$-module.
\end{proof}

Now, using the previous lemma jointly with Corollary \ref{corol weak acyclic finite relgldim as bimod} we obtain a class of strongly proj-bounded extensions.  

\begin{prop}\label{prop BA1A2 proj bounded extension example}
Let $X_1, Q, X_2$ be quivers with $X_1, X_2$ acyclic and $Q$ weakly triangular.  Define the algebras $A_i = k[X_i]/I_i$ with $I_i$ an admissible ideal of $k[X_i]$ for $i = 1,2$ and $\overline{B} = k[Q]/I_Q$ for $I_Q$ an admissible ideal of $k[Q]$.  Defining the algebras $B,A$ as before Lemma \ref{lemma matrix AoverB proj as bimod}, the extension $B\subseteq A$ is strongly proj-bounded.
\end{prop}

\begin{proof}
By Lemma \ref{lemma matrix AoverB proj as bimod} the extension is proj-bounded, so we need only check that $A$ has finite $B$-relative projective dimension as an $A$-bimodule, which will follow from Corollary \ref{corol weak acyclic finite relgldim as bimod} once we check that the hypotheses of Corollary \ref{corol weakly acyclic has finite relgd} are satisfied for this extension.  Label the $n_2$ vertices of $X_2$ by $1,\hdots, n_2$ so that there are arrows from $i\to j$ only when $i$ is greater than or equal to $j$, now label the $n$ vertices of $Q$ by $n_2+1,\hdots,n_2+n$  so that there are arrows from $i\to j$ only when $i$ is greater than or equal to $j$, and finally label the $n_1$ vertices of $X_1$ by $n_2+n+1,\hdots, n_2+n+n_1$ so that there are arrows from $i\to j$ only when $i$ is greater than or equal to $j$.  One may directly check by multiplying matrices that this is a complete set of idempotents for $A$ and that whenever $i<j$ in this ordering then $jai = 0$ for any element $a$ of $A$, as required.  It remains to check the conditions on the subalgebra $B$.  
The vertices of the required quiver $R$ can be chosen to be the vertices of $Q$ together with the sum $\sum_{x\in (X_1)_0\cup (X_2)_0}x$, so the first property is satisfied. The arrows of $R$ can be chosen to be the arrows of $Q$, and hence the second property is satisfied. It remains to check that $R$ contains every loop of the quiver of $A$, which is equivalent to saying that for any vertex $i$ of $X_1$ or $X_2$, $iAi = 0$.  But this is clear, since
$$
\begin{pmatrix}0 & 0 & 0 \\
0 & 0 & 0 \\
0 & 0 & i\end{pmatrix}
A
\begin{pmatrix}0 & 0 & 0 \\
0 & 0 & 0 \\
0 & 0 & i\end{pmatrix}
= 
\begin{pmatrix}0 & 0 & 0 \\
0 & 0 & 0 \\
0 & 0 & iA_1i\end{pmatrix} = 0
$$
because $A_1$ is acyclic, and similarly with $i\in A_2$.
\end{proof}

\begin{remark}
Note that in these examples, $A/B$ is projective as a $B$-bimodule, which is stronger than we require for an extension to be proj-bounded.  As mentioned in Section \ref{Section Ext.finite.gldim}, a deeper understanding of relative homological algebra for associative algebras is likely to allow larger classes of examples.
\end{remark}

\begin{remark}
In terms of quivers, the extensions of Proposition \ref{prop BA1A2 proj bounded extension example} have the following form:
$$
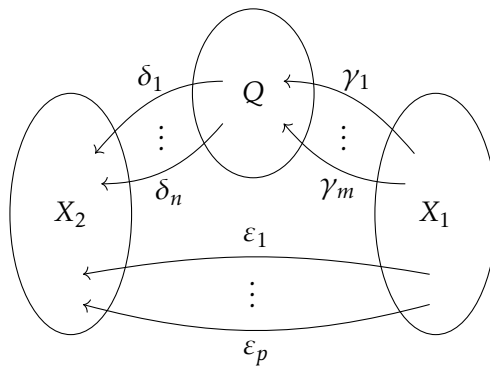
\begin{tikzpicture}[scale = 0.8]

\draw (0,0) ellipse (1cm and 2cm);

\draw (0,0) node {$X_2$};
\draw (3,2) node {$Q$};
\draw (6,0) node {$X_1$};

\draw (3,2) ellipse (1cm and 1.4cm);

\draw (6,0) ellipse (1cm and 2cm);

\draw[->] (2.5,2.2) to[in = 50, out = 180] node[midway, above] {$\delta_1$} (0.4,1);
\draw (1.5,1.4) node {$\vdots$};
\draw[->] (2.5,1.5) to[in = 0, out = -130] 
node[midway, below] {$\delta_n$} (0.5,0.5);

\draw[->] (5.65,1) to[in = 0, out = 130] node[midway, above] {$\gamma_1$} (3.5,2.2) ;
\draw (4.5,1.4) node {$\vdots$};
\draw[->] (5.5,0.5) to[in = -50, out = 180] 
node[midway, below] {$\gamma_m$} (3.5,1.5) ;

\draw[->] (5.9,-1) to[in = 10, out = 170] node[midway, above] {$\varepsilon_1$} (0.2,-1) ;
\draw (3,-1.2) node {$\vdots$};
\draw[->] (5.9,-1.5) to[in = -15, out = 195] node[midway, below] {$\varepsilon_p$} (0.2,-1.5) ;
\end{tikzpicture}
$$
The ovals represent quivers.  The quivers $X_1, X_2$ are acyclic and $Q$ is weakly triangular.  The algebra $A_1$ is $k[X_1]/I_1$ for $I_1$ an admissible ideal of $X_1$ and similarly $A_2$ is $k[X_2]/I_2$ and $\overline{B}$ is $k[Q]/I_Q$.  The subalgebra $B$ is $\overline{B}\times \langle E\rangle$, where $E$ is the sum of the vertex idempotents of $X_1$ and of $X_2$.  The bimodules $M_1, M_2$ are 
generated by the $\gamma_i$ and by the $\delta_i$, respectively.  More general bimodules are possible, but projectivity of $M_1$ (resp.\ $M_2$) as a left (resp.\ right) $\overline{B}$-module (cf.\ Lemma \ref{lemma matrix AoverB proj as bimod}) can be guaranteed by defining
$$M_1 = \sum_{i\in\{1,\hdots,m\}}\overline{B}\gamma_iA_1 \iso \bigoplus_{i\in\{1,\hdots,m\}}\overline{B}t(\gamma_i)\otimes_k s(\gamma_i)A_1$$
$$M_2 = \sum_{i\in\{1,\hdots,n\}}A_2\delta_i\overline{B}\iso \bigoplus_{i\in\{1,\hdots,n\}}A_2t(\delta_i)\otimes_k s(\delta_i)\overline{B},$$
where $s(\alpha), t(\alpha)$ represent the source and target vertices of the arrow $\alpha$.  
The bimodule $M_{21}$ is a bimodule quotient of $M_2\otimes_{\overline{B}}M_1 + \sum_{i\in \{1,\hdots,p\}}A_2\varepsilon_iA_1$, with the only rule being that any path appearing in a relation must be of length at least $2$ (since otherwise the ideal of the path algebra defining $A$ will not be admissible).
\end{remark}

\begin{remark}
The algebras of this construction may be compared with the closely related ``linearly ordered pullbacks'' considered in \cite{CoelhoWagnerLOPs}.  When $M_{21}=0$ the algebras of the present construction are examples of linearly ordered pullbacks.  Not every linearly ordered pullback is of the form of this construction however, because after translating between languages, the $M_1, M_2$ coming from a linearly ordered pullback need not be projective as (left and right respectively) $\overline{B}$-modules.
\end{remark}

\begin{example}
Follows an explicit example.  Consider $A$ to be the path algebra of the quiver
$$
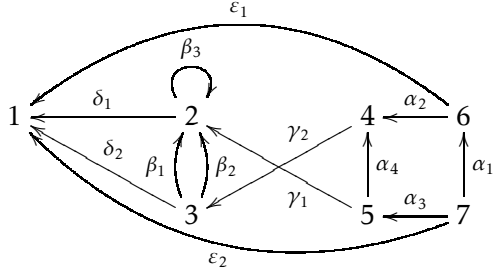
\xymatrix{
1 & & 2\ar@(ul,ur)^{\beta_3}\ar[ll]_{\delta_1} & & 4\ar[lld]_<<<<<<<{\gamma_2} & 6\ar[l]_{\alpha_2} \ar@/_{35pt}/[lllll]_{\varepsilon_1} \\
  & & 3\ar@/^/[u]^{\beta_1}\ar@/_/[u]_{\beta_2}\ar[llu]_{\delta_2} & & 5\ar[u]_{\alpha_4}\ar[llu]^<<<<<<<{\gamma_1} & 7\ar[l]_{\alpha_3}\ar[u]_{\alpha_1} \ar@/^{30pt}/[lllllu]^{\varepsilon_2} 
}$$
modulo the admissible ideal $I$ generated by the relations 
$$\alpha_2\alpha_1 = \alpha_4\alpha_3\,,\,{\beta_3}^3 = 0\,,\,\beta_3\beta_1 = \beta_3\beta_2\,,\,\delta_1\beta_3\gamma_1\alpha_3 = \delta_2\gamma_2\alpha_2\alpha_1\,,\,\varepsilon_1\alpha_1 = \delta_1\gamma_1\alpha_3$$
and $B$ to be the subalgebra with basis
$$\{e_2,\, e_3,\, e_1+e_4+e_5+e_6+e_7,\,\beta_1,\, \beta_2,\, \beta_3,\, \beta_3\beta_1,\, {\beta_3}^2,\, {\beta_3}^2\beta_1\}.$$
Proposition  \ref{prop BA1A2 proj bounded extension example} shows that the extension $B\subseteq A$ is strongly proj-bounded.

\end{example}

\section{Jacobi-Zariski sequences for proj-bounded extensions}\label{Section. JZ sequences}

A key result from \cite{CLMS20Arx} (see also \cite{CLMS20ArxCor} for a corrected formulation and \cite{CLMS24} for a corrected proof) is what they call a ``Jacobi-Zariski almost exact sequence'' for an arbitrary extension of algebras.  The result is then applied to bounded extensions $B\subseteq A$ in \cite[Theorem 6.5]{CLMS20Arx} (see also \cite[Theorem 6.5]{CLMS20ArxCor} and \cite[Theorem 3.8]{CLMS24}) to provide an exact sequence that compares the Hochschild homologies of $B$ and $A$, via the $B$-relative Hochschild homology of $A$.  Their argument applies equally well to the wider class of proj-bounded extensions:

\begin{theorem}\label{JZ proj-bounded}
Let $B \subseteq A$ be a proj-bounded extension of $k$-algebras and let $X$ be an $A$ bimodule. Assume that $A/B$ has index of projectivity $n$ and $\tn{pd}_{B^e}{A/B}=u$. Then there is a long exact sequence
\begin{align*}
\ldots \rightarrow \HH_{m}(B, X) \rightarrow \HH_{m}(A, X) \rightarrow \HH_{m}(A \mid B, X) \rightarrow \HH_{m-1}(B, X) \rightarrow \ldots\\
\rightarrow \HH_{n(u+1)}(B, X) \rightarrow \HH_{n(u+1)}(A, X) \rightarrow \HH_{n(u+1)}(A \mid B, X).
\end{align*}
\end{theorem}

\begin{proof}
Since $A/B$ is projective as either a left or a right $B$-module, it follows that 
$$\operatorname{Tor}_{*}^{B}\left(A / B,(A / B)^{\otimes_{B} n}\right)=0$$ 
for $*>0$ and for all $n$, and so \cite[Theorem 3.6]{CLMS24} applies. For degrees at least $2$ the terms on Page 1 of the spectral sequence which converge to the gap are
$$
E_{p, q}^{1}=\operatorname{Tor}_{q}^{B^{e}}\left(X,(A / B)^{\otimes_B p}\right) \quad \text { for } p, q>0
$$
and $0$ elsewhere.  Since $A/B$ is projective on one side, we have that $(A/B)^{\otimes_{B} p}$ is of projective dimension at most $pu$ (see \cite[Chapter IX, Proposition 2.6]{CartanEilenberg56}).
Now if $p+q \geqslant n(u+1)$, then 
$p\geqslant u$ or $q> pu$. In both cases
$E_{p, q}^{1}=0$. Consequently the gap is $0$ in degrees $\geqslant n$.
\end{proof}

The arguments used in \cite{CLMS20Arx,CLMS24} and Theorem \ref{JZ proj-bounded} are completely formal, and hence can be translated directly to pseudocompact algebras:

\begin{theorem}\label{JZsequence pc version}
Let $B \subseteq A$ be a proj-bounded extension of pseudocompact $k$-algebras and let $X$ be a pseudocompact $A$-bimodule. Assume that $A/B$ has index of projectivity $n$ and $\tn{pd}_{B^e}{A/B}=u$. Then there is a 
long exact sequence
\begin{align*}
\ldots \rightarrow \HH_{m}(B, X) \rightarrow \HH_{m}(A, X) \rightarrow \HH_{m}(A \mid B, X) \rightarrow \HH_{m-1}(B, X) \rightarrow \ldots\\
\rightarrow \HH_{n(u+1)}(B, X) \rightarrow \HH_{n(u+1)}(A, X) \rightarrow \HH_{n(u+1)}(A \mid B, X).
\end{align*}
\end{theorem}

\section{Homological properties preserved by strongly proj-bounded extensions}\label{Section.HomProperties}

 Assume throughout this section that $B\subseteq A$ is a strongly proj-bounded extension, and that $A/B$ is projective as a right (rather than a left) $B$-module.  We prove that strongly proj-bounded extensions preserve the finitude of the left global dimension, the left finitistic dimension (for arbitrary algebras this means Findim, but when $A$ is finite dimensional, finite findim is also preserved, see for instance \cite{ZimmermannHuisgenFindimsurvey} for definitions), and the vanishing of Hochschild homology.  When $X$ is an $A$-module, we sometimes use the notation ${}_BX $ to denote the restriction of $X$ to a $B$-module. We will prove the case of abstract extensions of algebras, but throughout, identical results (with almost identical proofs) apply for extensions of pseudocompact algebras.  See also \cite{QinXuZhangZhou} for recent and related homological results for a different generalization of bounded extensions of algebras.

\subsection{Preservation of the finitude of the left global and finitistic dimensions}

Certain arguments in this section are similar to arguments from \cite[Section 4]{CLMS21}.

\begin{lemma}\label{lemma pd of proj module restricted}
If $P$ is a projective $A$-module, then the projective dimension of ${}_BP $ is finite and bounded above by $\tn{pd}_B(A)$.
\end{lemma}

\begin{proof}
By hypothesis $A/B$ has finite projective dimension as a left $B$-module, and hence so does $A$, as can be seen via the exact sequence
$$0\to B\to A\to A/B\to 0.$$
Say that the projective dimension of $A$ as a left $B$-module is $n$.  Since any projective module is a direct summand of a direct sum of copies of $A$, then ${}_BP$ also has projective dimension at most $n$.
\end{proof}

\begin{lemma}\label{lemma pd of induced module}
Let $Y$ be a left $B$-module with finite projective dimension.  Then the projective dimension of $A\otimes_B Y$ as a left $A$-module  is limited above by the projective dimension of $Y$.
\end{lemma}

\begin{proof}
Let $P_*\to Y\to 0$ be a  projective resolution of $Y$ of length $\tn{pd}_B Y$.  Then the sequence $A\otimes_B P_*\to A\otimes_B Y\to 0$ is exact because $A$ is projective as a right $B$-module, and furthermore each module $A\otimes_B P_i$ is projective as a left $A$-module, because the $P_i$ are projective as $B$-modules.  Hence the sequence obtained is a projective resolution for $A\otimes_B Y$ of length $\tn{pd}_B Y$.
\end{proof}

\begin{lemma}\label{lemma tensor with proj bimod is proj}
If $P$ is a projective $B$-bimodule and $X$ is any left $B$-module, then $P\otimes_B X$ is projective as a left $B$-module.
\end{lemma}

\begin{proof}
The module $P$ is a direct summand of $\bigoplus_I(B\otimes_k B)$ for some indexing set $I$, by hypothesis.  Hence $P\otimes_B X$ is a direct summand of 
$$(\bigoplus_I(B\otimes_k B))\otimes_B X \iso \bigoplus_I(B\otimes_k B \otimes_B X) \iso \bigoplus_I(B\otimes_k X),$$
which is free as a left $B$-module.
\end{proof}

\begin{lemma}\label{lemma proj dim of AoverB tensor X}
If $X$ is any left $B$-module, then the left $B$-module $(A/B)\otimes_B X$ has projective dimension at most the projective dimension of $A/B$ as a $B$-bimodule.
\end{lemma}

\begin{proof}
Recall that by hypothesis $A/B$ is projective as a right $B$-module.  Consider a projective resolution of $A/B$ as a $B$-bimodule:
$$0\to P_n\to \cdots\to P_1\to P_0 \to A/B \to 0.$$
Since $A/B$ is projective as a right $B$-module, the sequence
$$0\to P_n\otimes_B X\to \cdots\to P_1\otimes_B X\to P_0\otimes_B X \to (A/B)\otimes_B X \to 0$$
is exact.
  But the modules $P_i\otimes_B X$ are projective as left $B$-modules by Lemma \ref{lemma tensor with proj bimod is proj}, and hence the projective dimension of $(A/B)\otimes_BX$ is at most $n$.
\end{proof}

\begin{lemma}\label{lemma bound on pd of restricted module}
If the $A$-module $X$ has finite projective dimension, then the module ${}_BX$ has projective dimension not more than $\tn{pd}_A X + \tn{pd}_B A$.
\end{lemma}

\begin{proof}
Let $P_*\to X\to 0$ be a projective resolution of $X$ of length $\tn{pd}_A X$.  Restrict this sequence to $B$, hence treating it as an exact sequence of $B$-modules.  By Lemma \ref{lemma pd of proj module restricted}, each module ${}_B(P_i)$ has projective dimension bounded above by $\tn{pd}_B A$.  It follows that the projective dimension of ${}_BX $ is bounded above by $\tn{pd}_A X + \tn{pd}_B A$.
\end{proof}

The next two propositions do not require the full strength of strongly proj-bounded extensions, see Remark \ref{remark quotient bifinite extensions}.

\begin{prop}\label{prop AfinfindimpliesBfinfindim}
If $A$ has finite left finitistic dimension, then so does $B$.
\end{prop}

\begin{proof}
Denote by $d_A$ the finitistic dimension of $A$.   Let $Y$ be a left $B$-module with finite projective dimension.  Since $A/B$ is projective as a right $B$-module, the sequence 
$$0 \to B\to A \to A/B\to 0$$
of right $B$-module homomorphisms is split, and hence 
$$0 \to B\otimes_B Y\to A\otimes_B Y \to (A/B)\otimes_B Y\to 0$$
is exact.  We will check that the projective dimension of the two terms on the right is limited above by a number independent of $Y$, and hence so is $Y \iso B\otimes_B Y$.

Since $Y$ has finite projective dimension, so does $A\otimes_B Y$ by Lemma \ref{lemma pd of induced module}.  Hence $A\otimes_B Y$ has a projective resolution $P_*\to A\otimes_B Y\to 0$ of length not more than $d_A$.  By Lemma \ref{lemma bound on pd of restricted module}, ${}_B(A\otimes_B Y)$ has projective dimension not more that $\tn{pd}_B A + d_A$.   The final term $(A/B)\otimes_B Y$ has projective dimension not more than the projective dimension of $A/B$ as a $B$-bimodule by Lemma \ref{lemma proj dim of AoverB tensor X}. 
\end{proof}

\begin{prop}\label{propAfinitegdimpliesBfinitegd}
If $A$ has finite left global dimension, then so does $B$.
\end{prop}

\begin{proof}
The proof is essentially identical to that of Proposition \ref{prop AfinfindimpliesBfinfindim}: denote by $d_A$ the left global dimension of $A$ and let $Y$ be a $B$-module.  Considering the same short exact sequence, ${}_B(A\otimes_B Y)$ has projective dimension not more than $\tn{pd}_B(A) + d_A$ by Lemma \ref{lemma bound on pd of restricted module}, and $(A/B)\otimes_B Y$ has projective dimension not more than the projective dimension of $A/B$ as a $B$-bimodule by Lemma \ref{lemma proj dim of AoverB tensor X}.
\end{proof}

\begin{prop}\label{propBfinfindimimpliesAfinfindim}
If $B$ has finite left finitistic dimension, then so does $A$.
\end{prop}

\begin{proof}
Let $X$ be an $A$-module with finite projective dimension.  Suppose that $(A/B)^{\otimes_B p}$ is projective as a $B$-bimodule and that $A$ has relative projective dimension less than $p$.  The kernel $K$ of the map
$$A\otimes_B (A/B)^{\otimes_B{p}}\otimes_B A\to 
A\otimes_B (A/B)^{\otimes_B{p-1}}\otimes_B A$$
in the relative projective resolution of $A$ given in \cite[Proposition 2.1]{CLMS21} is a direct summand of $A\otimes_B (A/B)^{\otimes_B{p}}\otimes_B A$ as an $A$-bimodule by Lemma \ref{lemma after relprojdim kernel is summand}. 

We have an exact sequence
$$0\to K \to A\otimes_B (A/B)^{\otimes_Bp}\otimes_B A \to \hdots\to A\otimes_B A\to A\to 0$$
which remains exact when we apply $-\otimes_A X$ because the contracting homotopies are right $A$-module maps.  So we get an exact sequence
$$0\to K\otimes_A X \to A\otimes_B (A/B)^{\otimes_Bp}\otimes_B X \to \hdots\to A\otimes_B X\to X\to 0.$$
If $X$ has finite projective dimension as an $A$-module then it has finite projective dimension as a $B$-module by Lemma \ref{lemma bound on pd of restricted module}, and hence ${}_BX$ has projective dimension limited above by the finitistic dimension of $B$.  

Note that each module in this sequence except $X$ and $K\otimes_A X$ is of the form $A\otimes_B Y$ for some $B$-module $Y$ of finite projective dimension, by Lemma \ref{lemma proj dim of AoverB tensor X} (for the module $A\otimes_B X$ we use that ${}_BX$ is of finite projective dimension).  Note also that the projective dimension of each $Y$ is limited above by $\tn{Findim}(B)$.  From Lemma \ref{lemma pd of induced module} it follows that every module in the sequence except perhaps $X$ and $K\otimes_A X$ has projective dimension limited above by $\tn{Findim}(B)$.  It remains to check that the same is true for $K\otimes_B X$.  But $K$ is a direct summand of $A\otimes_B (A/B)^{\otimes_Bp}\otimes_B A$ and hence $K\otimes_A X$ is a direct summand of $(A\otimes_B (A/B)^{\otimes_Bp}\otimes_B A)\otimes_AX$.  So $K\otimes_A X$ also has projective dimension at most $\tn{Findim}(B)$ by the same argument.  

Putting all this together, $X$ has a resolution of length limited above by the value such that $(A/B)^{\otimes_B p}$ is projective as a $B$-bimodule, and the relative projective dimension of $A$.  Each module in the resolution has projective dimension limited above by the finitistic dimension of $B$.  In particular, the projective dimension of $X$ is limited independent of $X$, as required.

\end{proof}

\begin{prop}\label{propBfinitegdimpliesAfinitegd}
If $B$ has finite left global dimension, then so does $A$.
\end{prop}

\begin{proof}
This can be proved exactly as Proposition \ref{propBfinfindimimpliesAfinfindim}, by swapping the left finitistic dimension of $B$ with the left global dimension of $B$ throughout.  The only difference is that several claims proved in Proposition \ref{propBfinfindimimpliesAfinfindim} are trivial when $B$ has finite left global dimension.
\end{proof}

\begin{theorem}\label{theorem abstract preservation of gd and fd}
Let $B\subseteq A$ be a strongly proj-bounded extension of algebras.
\begin{enumerate}
    \item $B$ has finite left global dimension if, and only if, $A$ does.
    \item $\tn{Findim}(B)$ is finite if, and only if, $\tn{Findim}(A)$ is finite.
\end{enumerate}
\end{theorem}

\begin{proof}
Part 1 is Propositions \ref{propAfinitegdimpliesBfinitegd}, \ref{propBfinitegdimpliesAfinitegd} and Part 2 is Propositions \ref{prop AfinfindimpliesBfinfindim}, \ref{propBfinfindimimpliesAfinfindim}.
\end{proof}

\begin{theorem}\label{Th.FindimFin}
If $A$ is finite dimensional, then $\tn{findim}(B)$ is finite if, and only if, $\tn{findim}(A)$ is finite.
\end{theorem}

\begin{proof}
If $A$ is finite dimensional, then the restriction of a finitely generated module to $B$ is finitely generated.  Thus all the modules appearing in the proofs of Propositions \ref{prop AfinfindimpliesBfinfindim}, \ref{propBfinfindimimpliesAfinfindim} may be assumed to be finitely generated.
\end{proof}

\begin{remark}
In \cite{XiXu} and in
  \cite{Guo2} (see also \cite{Xi04,Xi06, Xi08}) the authors study the finitistic dimension conjecture for extensions of Artin algebras. In particular, a new formulation of the finitistic dimension conjecture in terms
  of relative homological dimension is given.
\end{remark}

\begin{remark}
Reduction techniques to attack the finitistic dimension conjecture already appear in the literature. For instance, in  \cite{GPS}, viewing a finite dimensional algebra as a quotient of a path algebra,  the authors present two operations on the quiver of the algebra, namely arrow removal and vertex removal, and show that these operations preserve the finitude of the finitistic dimension.  That arrow removal preserves $\tn{findim}$ \cite[Theorem A]{GPS} follows from Theorem \ref{Th.FindimFin}, as one may observe that the corresponding extension of algebras is strongly proj-bounded. 
\end{remark}

\begin{remark}\label{remark quotient bifinite extensions}
The full strength of strongly proj-bounded extensions was used to show that when $B$ has finite finitistic dimension, then so does $A$. Analysing the proof of Proposition \ref{prop AfinfindimpliesBfinfindim} one observes that less is required to show the other implication, and indeed it was recently shown that less still is required, at least for finite dimensional algebras \cite{FJ}.
\end{remark}

If $A$ is a pseudocompact algebra, the (big) left finitistic dimension of $A$ is defined in the obvious way: as the supremum of the projective dimensions of those pseudocompact left $A$-modules having finite projective dimension.

\begin{theorem}\label{theorem pc preservation of gd and fd}
Let $B\subseteq A$ be a strongly proj-bounded extension of pseudocompact algebras.
\begin{enumerate}
    \item $B$ has finite global dimension if, and only if, $A$ does.
    \item $B$ has finite left finitistic dimension if, and only if, $A$ does.
\end{enumerate}
\end{theorem}

\begin{proof}
The arguments of this section go through making only superficial changes (tensors are completed tensors, free modules are products rather than sums, etc.).  The left and right global dimensions of a pseudocompact algebra coincide by an observation of Brumer \cite[p. 449]{Brumer}.
\end{proof}

\subsection{Preservation of vanishing of $\tn{HH}$ and Han's Conjecture}

\begin{theorem}\label{Theorem.HHhom}
Let $B\subseteq A$ be a strongly proj-bounded  extension.  Then $\HH_m(A)$ vanishes for large enough $m$ if, and only if, $\HH_m(B)$ vanishes for large enough $m$.
\end{theorem}

\begin{proof}
    We have $\HH_r(A|B,A) = 0$ for any $r$ bigger than $\mathrm{pd}_{(A^e,B^e)}A$, which is finite by hypothesis, and so Theorem \ref{JZ proj-bounded} implies that $\HH_m(B,A)\iso \HH_m(A)$ for sufficiently large $m$.  On the other hand, applying the functor $-\otimes_{B^e}B$ to the sequence
    $0\to B \to A \to A/B\to 0,$
    one obtains the long exact sequence
    $$\cdots\to \HH_{m+1}(B,A/B)\to \HH_m(B)\to \HH_m(B,A)\to \HH_m(B,A/B)\to \cdots,$$
    and hence $\HH_m(B)\iso \HH_m(B,A)$ for any $m$ greater than $\mathrm{pd}_{B^e}A/B$.  Putting this together, 
    $$\HH_m(A) \iso \HH_m(B)$$
    for large enough $m$ and in particular, one is $0$ if, and only if, the other is.
\end{proof}

\begin{theorem}\label{Theorem.HHhomPC}
Let $B\subseteq A$ be a strongly proj-bounded extension of pseudocompact algebras.  Then $\HH_m(A)$ vanishes for large enough $m$ if, and only if, $\HH_m(B)$ vanishes for large enough $m$.
\end{theorem}

\begin{proof}
The proof is just as for Theorem \ref{Theorem.HHhom}, using Theorem \ref{JZsequence pc version} instead of Theorem \ref{JZ proj-bounded}.
\end{proof}

Recall that Han's conjecture \cite[Conjecture 1]{Han06} for finite dimensional algebras asserts that if $A$ is a finite dimensional algebra, then $\tn{HH}_m(A)$ vanishes for sufficiently large $m$ if, and only if, $A$ has finite global dimension.  The following result generalizes \cite[Theorem 4.6]{CLMS20-}.

\begin{corol}\label{corol abstract han}
If $B\subseteq A$ is a strongly proj-bounded extension of finite dimensional algebras, then Han's conjecture holds for $A$ if, and only if, it holds for $B$.
\end{corol}

\begin{proof}
This is immediate from Theorems \ref{theorem abstract preservation of gd and fd} and \ref{Theorem.HHhom}.
\end{proof}

\begin{remark}
Of course, a similar statement holds for more general algebras, but for general pseudocompact algebras, Han's conjecture is false.  For example, consider the completed path algebra of the infinite quiver
 \[Q = \begin{tikzcd}
 0 & 1 \arrow[l] & 2 \arrow[l] & \cdots \arrow[l] &  
 \end{tikzcd}\]
and the pseudocompact algebra $A = k\db{Q}/J^2$, where $J^2$ is the closed ideal generated by the paths of length $2$.  Then the simple left $A$-module $S_n$ at vertex $n$ has projective dimension $n$, and hence the global dimension of $A$ is infinite.  Meanwhile, $\tn{HH}_*(A)$ is concentrated in degree $0$, as can be seen using Proposition \ref{prop.HomologyLimit}.  The quiver extending infinitely in the other direction is even worse: we still have $\tn{HH}_*(A)$ concentrated in degree $0$, but in this algebra every simple left module has infinite projective dimension.

While Han's Conjecture is false for general pseudocompact algebras, it may hold for interesting special classes.  For example, it holds for the completed group algebra of certain profinite groups \cite[Section 5]{G23}.
\end{remark}

 \newcommand{\noop}[1]{}

\end{document}